\documentclass[a4paper,12pt]{amsart}
\usepackage{amsmath,a4wide,amsfonts,amsthm,amsopn,amssymb}%cite,mathrsfs}
\usepackage{graphicx}
\usepackage[english]{babel}
\usepackage[latin1, utf8]{inputenc}

\usepackage{mathtools}
\mathtoolsset{showonlyrefs}

\usepackage{epsfig,verbatim}
\usepackage{mathrsfs}
\usepackage{hyperref}
\usepackage{amsaddr}

\newtheoremstyle{mythm}{6pt}{6pt}{\itshape}{}{\bfseries}{.}{ }{} %Format for Theorem
\newtheoremstyle{mydef}{6pt}{6pt}{}{}{\bfseries}{.}{ }{}         %Format for Definition
\newtheoremstyle{myrem}{6pt}{6pt}{}{}{\scshape}{.}{ }{}          %Format for Remark

\theoremstyle{mydef}
\newtheorem{definition}{Definition}[section]

\theoremstyle{mythm}
\newtheorem{theorem}[definition]{Theorem}
\newtheorem{proposition}[definition]{Proposition}
\newtheorem{lemma}[definition]{Lemma}

\theoremstyle{myrem}

\renewenvironment{proof}{\par\medbreak\noindent\textit{Proof}:\hskip.5em\ignorespaces}{\hfill\qedsymbol\medbreak} %Format of proof
\renewcommand{\qedsymbol}{\hfill\rule{2mm}{2mm}}                                                                  %End of proof symbol is a black square
\newcommand{\R}{\mathbb{R}}

\newcommand{\gs}{\varphi_0}

\def\cI{\mathcal{I}}
\def\cV{\mathcal{V}}

\providecommand{\norm}[1]{\lVert#1\rVert}
\newcommand{\N}{\mathbb{N}}
\newcommand{\E}{\mathcal{E}}
\newcommand{\V}{\mathcal{V}}
\newcommand{\Lt}[1][d]{L^2(\R^{#1})}
\newcommand{\G}{\mathcal{G}}

\renewcommand{\l}{\lambda}

\newcommand{\Chi}{\protect\raisebox{2.5pt}{$\chi$}}

\allowdisplaybreaks

\begin{document}
\title{Multi-Window Weaving Frames}
\author{Monika Dörfler
\and Markus Faulhuber}
\address{NuHAG, Faculty of Mathematics, University of Vienna\\Oskar-Morgenstern-Platz 1, 1090 Vienna, Austria}

\thanks{M.D.\ was supported by the Vienna Science and Technology Fund (WWTF): project SALSA (MA14-018); M.F.\ was supported by the Austrian Science Fund (FWF): [P26273-N25] and [P27773-N23].}
\keywords{SampTA 2017, Gabor Frames, Localization Operators, Multi-Window Frames, Weaving Frames}

\begin{abstract}
	In this work we deal with the recently introduced concept of weaving frames. We extend the concept to include multi-window frames and present the first sufficient criteria for a family of multi-window Gabor frames to be woven. We give a Hilbert space norm criterion and a pointwise criterion in phase space. The key ingredient are localization operators in phase space and we give examples of woven multi-window Gabor frames consisting of Hermite functions.
\end{abstract}

\maketitle

\section{Introduction}
Recently, the concept of {\it weaving frames} has been introduced in \cite{becagrlaly16, caly15}. The motivation of this concept is found in distributed signal processing. We are given one sampling pattern and two sets of linear measurements, labelled with respect to the pattern, which both allow for a stable reconstruction of a signal, i.e.\ we have two frames. If at any sampling point we can choose one or the other method of measuring, i.e.\ at each point we choose one or the other corresponding frame element, and for any such choice we still get a stable reconstruction, i.e.\ a frame, we call the two systems woven. This concept can of course be extended to a larger (even countable) number of frames.

We give some basic examples. Of course any frame is woven with itself, however, the labelling of the frame elements is important. Let $\Phi = \lbrace \phi_1, \phi_2 \rbrace$ and $\Psi = \lbrace \psi_1, \psi_2 \rbrace$ with $\phi_1 = \psi_2 = e_1$ and $\phi_2 = \psi_1 = e_2$ where $\lbrace e_1, e_2 \rbrace$ is the standard basis for $\R^2$. Then, both $\Phi$ and $\Psi$ form an orthonormal basis for $\R^2$, but they are not woven since e.g.\ $\lbrace \phi_1, \psi_2 \rbrace$ does not span $\R^2$.

While several characterizations for families of frames to be woven have been given in \cite{becagrlaly16, caly15}, more constructive examples are rare. In particular, one of the motivating questions, namely, whether Gabor frames generated by rotated general Gaussian windows are woven, has not been answered to date. A rotated general Gaussian window is a Gaussian function whose phase space concentration is described by a 2-dimensional Gaussian which is essentially supported on a rotated ellipse. We give a more formal definition at the end of this work in equation \eqref{eq_chriped_Gauss}.

We consider a more general notion, namely multi-window weaving frames and construct particular examples based on the eigenfunctions of so-called localization operators. As an application we get that multi-window Gabor frames with Hermite functions are woven if a sufficiently high number of generating windows is chosen.
We proceed by introducing some basic concepts for the Hilbert space $\Lt[]$. Although the statements can be generalized verbatim to higher dimensions, this work does not gain any deeper insights by using a more general notation.

\subsection{Gabor Frames and Weaving Frames}

We denote by $\gamma = (x, \omega) \in \R \times \R$ a point in the \textit{time-frequency plane}, also called \textit{phase space}, and use the following notation for a \textit{time-frequency shift} by $\gamma$:
\begin{equation}\label{eq_TFShift}
	\pi(\gamma)\phi(t) = M_\omega T_x \, \phi(t) = e^{2 \pi i \omega \cdot t} \phi(t-x),
\end{equation}
for  $ x,\omega,t \in \R$, where $T_x \phi(t) = \phi(t-x)$ and $M_\omega \phi(t) = e^{2 \pi i \omega \cdot t} \phi(t)$ are the translation and  the modulation operator, respectively.

A \textit{Gabor system} for $\Lt[]$ is generated by a \textit{window function} $\phi \in \Lt[]$ and an \textit{index set} $\Gamma \subset \R^{2}$. It is denoted by
\begin{equation}
	\G(\phi,\Gamma) = \lbrace \pi(\gamma)\phi \, | \, \gamma \in \Gamma \rbrace.
\end{equation}
$\mathcal{G}(\phi,\Gamma)$ is called a \textit{frame} if for all $f \in \Lt[]$
\begin{equation}\label{Frame}
	A \norm{f}_2^2 \leq \sum_{\gamma \in \Gamma} \left| \langle f, \pi(\gamma) \phi \rangle \right|^2 \leq B \norm{f}_2^2,
\end{equation}
with $0< A \leq B < \infty$ called \textit{frame bounds}.
\begin{definition}
	We call a finite family of frames, $\G(\phi_j,\Gamma)_{j = 1}^M$, \textit{woven} if for every partition $\{\sigma_j\}_{j=1}^M$ of $\Gamma$, the family $\{\G(\phi_j,\sigma_j (\Gamma ))\}_{j = 1}^M$ is a frame for $\Lt[]$.
\end{definition}

We will sometimes use the  notation $f \approx g$ if there exist constants $0 < A \leq B < \infty$ such that $A g \leq f \leq B g$.

\section{Localization Operators}\begin{definition}
For a window $\phi \in \Lt[]$, the short-time Fourier transform (STFT) at $\gamma= (x,\omega) \in \R \times \R$ of a function $f \in \Lt[]$ is defined as $\V_\phi f(\gamma) = \langle f, \pi(\gamma)\phi\rangle$ and for $m \in L^\infty(\R^2)$ the \textit{time-frequency localization operator} with symbol $m$ for $f \in \Lt[]$ is given by
	\begin{equation}
		H_{m,\phi} f(t)= \int_{\R^2} \V_\phi f(\gamma) m(\gamma)\pi(\gamma) \phi(t) d\gamma .	
	\end{equation}
	Formally we have $H_{m,\phi} f = \V^*_\phi m \V_\phi f$.
\end{definition}

We will be interested in families of localization operators, which cover the entire time-frequency domain, such that we can derive local windows from their eigenfunctions. We require the following property.
\begin{definition}
	A family of symbols $\lbrace \eta_\gamma \colon \R^{2} \to \R \, | \, \gamma \in \Gamma, \eta_\gamma \in L^1\left( \R^{2} \right) \rbrace$ is called \textit{well-spread} if $\Gamma \subset \R^{2}$ is a discrete set without accumulation points and there exists a continuous function $g \in L^2(\R^2)$ with polynomial decay, such that $|\eta_\gamma(\xi)| \leq g(\xi-\gamma)$ for $\xi \in \R^{2}$, $\gamma \in \Gamma$.
\end{definition}
Note that for localization operators as defined above, well-spread symbol families  always lead to well-spread operator families in  the sense of \cite{doro14}, i.e., for all $z\in \mathbb{R}$, $\gamma\in\Gamma$ we have the point-wise bound 
\begin{equation}
	|\V_\phi H_{\eta_\gamma,\phi}f(z)|\leq |T_\gamma g\cdot \V_\phi f|\ast |\V_\phi \phi | (z).
\end{equation}

We will now have a closer look at the time-frequency localization operator $H_{\eta_\gamma} \colon \Lt[] \to \Lt[]$. Assuming that the family of symbols $\lbrace \eta_\gamma \, | \, \gamma \in \Gamma \rbrace$ is non-negative and well-spread we can use the results in \cite{doro14} which state that, under these conditions, the localization operator $H_{\eta_\gamma}$ is positive and trace class and, hence, can be diagonalized. Therefore, we have
\begin{equation}
	H_{\eta_\gamma,\phi} f = \sum_{k \geq 1} \l_k^\gamma \left \langle f, \phi_k^\gamma \right \rangle \phi_k^\gamma, \quad f \in \Lt[],
\end{equation}
where $\left\lbrace \phi_k^\gamma \, | \, k \in \N \right\rbrace$ is an orthonormal subset of $\Lt[]$ consisting of eigenfunctions of $H_{\eta_\gamma}$. The sequence of eigenvalues $\left( \l_k^\gamma \right)_{k=1}^\infty$ is non-increasing sequence with non-negative real numbers. We define a related operator by putting a threshold on the eigenvalues. For $\varepsilon > 0$ we define
\begin{equation}
	H_{\eta_\gamma,\phi}^\varepsilon f = \sum_{k: \, \l_k^\gamma > \varepsilon} \l_k^\gamma \left \langle f, \phi_k^\gamma \right \rangle \phi_k^\gamma, \quad f \in \Lt[].
\end{equation}
Considering  the derived first $N_\varepsilon^\gamma$ eigenfunctions and observing that they are maximally concentrated within the support of $\eta_\gamma$ in the sense that among all orthonormal sets of functions, they maximize the quantity 
$\sum_{j = 1}^{N_\varepsilon^\gamma}\int_z \eta_\gamma (z) |\V_\phi\phi_j (z)|^2 dz$, motivates an approach that uses these windows as basic windows generating a {\it multi-window} Gabor frame, \cite{zezi97}. A multi-window Gabor system is, by obvious generalization, defined as the set of functions 
$ \lbrace \pi(\gamma)\phi_k\, | \, \gamma \in \Gamma, k\in\mathcal{I} \rbrace$ and denoted by $\G( \lbrace\phi_k \rbrace_{k\in\mathcal{I}},\Gamma)$.
\begin{definition}
	A family of multi-window Gabor frames $\G( \lbrace\phi_k^j \rbrace_{k\in\mathcal{I} },\Gamma)_{j = 1}^M$ is \textit{woven} if for every partition $\{\sigma_j\}_{j=1}^M$ of $\Gamma$, the family $\{\G(\lbrace\phi_k^j \rbrace_{k\in\mathcal{I} },\sigma_j (\Gamma ))\}_{j = 1}^M$ is a frame for $\Lt[]$.
\end{definition}
\section{Varying the localization window}
In a series of papers, cf.~\cite{doro14} for references, it has been shown, that, if a  family of localization operators $H_{\eta_\gamma,\phi}$ is {\it well-spread} and $\sum H_{\eta_\gamma,\phi}$ is invertible, one has the norm equivalence $\norm{f}_2^2 \approx \sum_{\gamma \in \Gamma} \norm{H_{\eta_{\gamma},\phi} f}_2^2$. Thereby, however, the window $\phi$ defining the localization operator remains the same for all $\gamma$. In the next sections, we show how the window can be varied. The rationale behind this approach is to generate different suitable families of eigenfunctions. We will then see, that, by picking finite sets from the operators' eigenfunctions, we obtain multi-window Gabor frames, and eventually, due to the different analysis windows,  the resulting different multi-window frames will be woven. We proceed to give two conditions on varying analysis windows, for which the overall family of localization operators remains invertible.
\subsection{Norm-Conditions}
\begin{definition}
	The \textit{time-frequency concentration} $\E_{\Omega,\phi}f$ of $f \in \Lt[]$ w.r.t.~$\phi \in \Lt[]$ is given by
	\begin{equation}
		\E_{\Omega,\phi}f = \int_{\Omega} |\V_\phi f|^2 \, d \l = \langle H_{\Omega,\phi}f,\, f \rangle.
	\end{equation}
\end{definition}

\begin{lemma}\label{lem_Localization_2Windows_Difference}
	Let $\Omega \subset \R^2$ and $\phi_1$, $\phi_2$ be two window functions. Then, for every $f \in \Lt[]$ we get the estimate
	\begin{equation}
		\begin{aligned}
			\left|\E_{\Omega,\phi_1}f - \E_{\Omega,\phi_2}f \right|
			\leq \left(\norm{\phi_1}_2 + \norm{\phi_2}_2 \right) \norm{\phi_1-\phi_2}_2 \norm{f}_2^2.
		\end{aligned}
	\end{equation}
\end{lemma}

\begin{proof}
	First, we note that we can write $H_{\Omega,\phi} = \V^*_\phi \Chi_\Omega \V_\phi$. We get
	\begin{equation}
		\begin{aligned}
			H_{\Omega,\phi_1}-H_{\Omega,\phi_2} = & \V^*_{\phi_1} \Chi_\Omega \V_{\phi_1} - \V^*_{\phi_2} \Chi_\Omega \V_{\phi_2}\\
			= & \V^*_{\phi_1} \Chi_\Omega \V_{\phi_1} - \V^*_{\phi_2} \Chi_\Omega \V_{\phi_2} + \V^*_{\phi_2} \Chi_\Omega \V_{\phi_1} - \V^*_{\phi_2} \Chi_\Omega \V_{\phi_1}\\
			= & \V^*_{\phi_1-\phi_2} \Chi_\Omega \V_{\phi_1} - \V^*_{\phi_2} \Chi_\Omega \V_{\phi_2-\phi_1}\\
			= & \V^*_{\phi_1-\phi_2} \Chi_\Omega \V_{\phi_1} + \V^*_{\phi_2} \Chi_\Omega \V_{\phi_1-\phi_2}
		\end{aligned}
	\end{equation}
	Now we compute
	\begin{equation}
		\begin{aligned}
			\left|\langle H_{\Omega,\phi_1}f,\, f \rangle - \langle H_{\Omega,\phi_2}f,\, f \rangle \right|
% 			& = \left| \left\langle \left(\V^*_{\phi_1-\phi_2} \Chi_\Omega \V_{\phi_1} + \V^*_{\phi_2} \Chi_\Omega \V_{\phi_1-\phi_2} \right) f, \, f \right\rangle \right|\\
			\leq & \left| \left\langle \V^*_{\phi_1-\phi_2} \Chi_\Omega \V_{\phi_1}f, f \right\rangle \right| + \left| \left\langle \V^*_{\phi_2} \Chi_\Omega \V_{\phi_1-\phi_2} f, f \right\rangle \right|\\
			\leq & \big(\norm{\phi_1}_2+\norm{\phi_2}_2 \big) \norm{\phi_1-\phi_2}_2 \norm{f}_2^2.
		\end{aligned}
	\end{equation}
\end{proof}

\begin{proposition}\label{prop_norm_equivalence}
	Let $\phi_1, \, \phi_2 \in \Lt[]$ with $\norm{\phi_i}_2 = 1$. Let $\lbrace \eta_{\gamma_1} \, | \, \gamma_1 \in \Gamma_1 \rbrace \cup \lbrace \eta_{\gamma_2} \, | \, \gamma_2 \in \Gamma_2 \rbrace$ be a well-spread family of non-negative symbols on $\R^2$ with $\sum_{i=1}^2 \sum_{\gamma_i \in \Gamma_i} \eta_{\gamma_i} \approx 1$. If $\norm{\phi_1-\phi_2}_2 < \frac{1}{2}$, then the following holds.
	\begin{equation}
		A \norm{f}_2^2 \leq \sum_{i=1}^2 \sum_{\gamma_i \in \Gamma_i} \norm{H_{\eta_{\gamma_i},\phi_i} f}_2^2 \leq B \norm{f}_2^2.
	\end{equation}
\end{proposition}

\begin{proof}
	We define $m_i = \sum_{\gamma_i} \eta_{\gamma_i}$ and $\text{supp}(m_i) = \Omega_i$. We get
	\begin{equation}
		\begin{aligned}
			& \sum_{i=1}^2 \left\langle \sum_{\gamma_i \in \Gamma_i} H_{\eta_{\gamma_i},\phi_i} f, \, f \right\rangle
			= \sum_{i=1}^2 \left\langle H_{m_i,\phi_i} f, \, f \right\rangle
			= \sum_{i=1}^2 \left\langle \V_{\phi_i}^* (m_i \V_{\phi_i}) f, \, f \right\rangle\\
			= & \sum_{i=1}^2 \left\langle m_i \V_{\phi_i} f, \, \V_{\phi_i} f \right\rangle
			= \sum_{i=1}^2 \int_{\Omega_i} m_i |\V_{\phi_i} f|^2 \, d\l
			\geq \, A \left(|\E_{\Omega_1,\phi_1} f| + |\E_{\Omega_2,\phi_2} f|\right)
		\end{aligned}
	\end{equation}
	with $A > 0$ as $\sum_{i=1}^2 \sum_{\gamma_i \in \Gamma_i} \eta_{\gamma_i} \approx 1$. From this assumption we also conclude that $\Omega_1 \cup \Omega_2 = \R^2$ and we have
	\begin{equation}
		\begin{aligned}
			\norm{f}_2^2 \leq |\E_{\Omega_1,\phi_1}f + \E_{\Omega_2,\phi_1}f|
% 			= & |\E_{\Omega_1,\phi_1}f + \E_{\Omega_2,\phi_1}f + \E_{\Omega_2,\phi_2}f - \E_{\Omega_2,\phi_2}f|
			\leq |\E_{\Omega_1,\phi_1}f| + |\E_{\Omega_2,\phi_2}f| + |\E_{\Omega_2,\phi_1}f - \E_{\Omega_2,\phi_2}f|.
		\end{aligned}
	\end{equation}
	Therefore, it follows from Lemma \ref{lem_Localization_2Windows_Difference} that
	\begin{equation}
		\begin{aligned}
			\underbrace{\left(1 - 2 \norm{\phi_1 - \phi_2}_2 \right)}_{C_{\phi_1,\phi_2}} \norm{f}_2^2 \leq |\E_{\Omega_1,\phi_1}f| + |\E_{\Omega_2,\phi_2}f|.
		\end{aligned}
	\end{equation}
	$C_{\phi_1,\phi_2} > 0$ if $\norm{\phi_1-\phi_2}_2 < \frac{1}{2}$. Hence, we get that
	\begin{equation}
		\sum_{i=1}^2 \left\langle \sum_{\gamma_i \in \Gamma_i} H_{\eta_{\gamma_i},\phi_i} f, \, f \right\rangle \geq A C_{\phi_1,\phi_2} \norm{f}_2^2.
	\end{equation}
	It is immediate that the upper bound is finite.
% 	Since $\phi_i \in \Lt[]$ there exists a constant $0 < B < \infty$ such that
% 	\begin{equation}
% 		 \sum_{i=1}^2 \sum_{\gamma_i \in \Gamma_i} \norm{H_{\eta_{\gamma_i},\phi_i} f}_2^2 \leq B \norm{f}_2^2.
% 	\end{equation}
% 	Set $\Gamma = \Gamma_1 \oplus \Gamma_2$. Then
% 	\begin{equation}
% 		\sum_{\gamma_i \in \Gamma_i} \norm{H_{\eta_{\gamma_i},\phi_i} f}_2^2 \leq \sum_{\gamma_i \in \Gamma} \norm{H_{\eta_{\gamma_i},\phi_i} f}_2^2 \leq B_i \norm{f}_2^2.
% 	\end{equation}
% 	Now define $B = 2 \, \max \lbrace B_1, B_2 \rbrace$ and we get
% 	\begin{equation}
% 		A C_{\phi_1,\phi_2} \norm{f}_2^2 \leq \sum_{i=1}^2 \sum_{\gamma_i \in \Gamma_i} \norm{H_{\eta_{\gamma_i},\phi_i} f}_2^2 \leq B \norm{f}_2^2.
% 	\end{equation}
\end{proof}
 \subsection{Phase-Space Conditions}
We will now establish sufficient conditions by pointwise estimates in phase space. Let
\begin{equation}
	\varphi_0(t) = 2^{1/4} e^{-\pi t^2}
\end{equation}
be the standard Gaussian of $L^2$-unit norm. In this section, we consider a collection of  windows $\Phi = \{\phi_i, \, i \in \cI\}$ such that 
\begin{equation}\label{cond1}
	| \cV_{\gs} \phi_i | \ast  |\cV_{\gs} \phi_i|(z) \leq C(1+|z|^2)^{-s}
\end{equation}
for all $i \in \cI$, $s>1$ and a constant $C$. Also, we let the symbols $\lbrace \eta_\gamma \colon \R^2 \to \R \, | \, \gamma \in \Gamma, \eta_\gamma \in L^1\left( \R^2 \right) \rbrace$ be well-spread with  $\sum_{\gamma \in \Gamma} \eta_\gamma (z) \approx 1$ and a window $ \phi_{i(\gamma )}$ from $\Phi$ be chosen for every index $\gamma \in \Gamma$.  We
consider the family of localization operators $H_{\eta_\gamma, \phi_{i(\gamma)}}$, i.e. at each $\gamma$ we allow the window to be picked from the  collection $\Phi = \{\phi_i, \, i \in \cI\}$.
We will need the following lemma.
\begin{lemma}\label{Lem_WS}
The family of localization operators $H_{\eta_\gamma, \phi_{i(\gamma)}}$, is well-spread, i.e. $\forall z \in \R^2$
	$$|\V_\phi H_{\eta_\gamma, \phi_{i(\gamma)}}f(z)|\leq |T_\gamma g\cdot \V_\phi f|\ast |\V_{\varphi_0} \varphi_0 |(z).$$
\end{lemma}

\begin{proposition}\label{Prop_diffwin}
	Consider a collection of  windows $\Phi  = \{\phi_i, \, i \in \cI\}$ with \eqref{cond1} and such that $| \cV_{\gs} \gs (z)- \cV_{\gs} \phi_i (z)| <C_0 | \cV_{\gs} \gs (z)|$ for all $i \in \cI$ and $2 \, C_0< \|  \V_{\gs} \gs \|_1$. Furthermore consider a well-spread family of symbols $\{\eta_\gamma\}_{\gamma \in \Gamma}$, Then, the following inequalities hold for some positive constants $A$, $B$ and all $f\in L^2(\R )$:
	\begin{equation}\label{OpFI}
		A\| f\|_2^2 \leq \sum_{\gamma \in \Gamma} \| H_{\eta_\gamma, \phi_{i(\gamma)}} f\|_2^2\leq B\| f\|_2^2 \, .
	\end{equation}
\end{proposition}
\begin{proof}
	We start by recalling from \cite{doro14}, that for \eqref{OpFI} to hold, the family of operators need to fulfill two conditions. First, they must be well-spread and second, the sum of operators must be invertible, i.e., we require, for some $A>0$, that
	\begin{equation}\label{Opinv}
		\left\langle \sum_{\gamma \in \Gamma} H_{\eta_\gamma, \phi_{i(\gamma)}} f, f \right\rangle \geq A \| f\|_2^2 \, .
	\end{equation}
	Well-spreadness is stated in Lemma \ref{Lem_WS}. We proceed to prove equation \eqref{Opinv}. First note that we trivially have
	\begin{equation*}
		\left\langle \sum_{\gamma \in \Gamma} H_{\eta_\gamma, \gs} f, f \right\rangle = \underbrace{\sum_{\gamma \in \Gamma} \int_{\eta_\gamma} |\cV_{\gs } f (z)|^2 dz}_{= \| f\|_2^2}.
	\end{equation*}
	For $\widetilde{f},\widetilde{g} \in L^1\left(\R^2\right)$, the twisted convolution $\widetilde{f} \natural \widetilde{g}$ is given by
	\begin{equation}
		\widetilde{f} \natural \widetilde{g}(x,\omega) = \int_{\R^2} \widetilde{f}(x',\omega') \widetilde{g}(x-x',\omega-\omega') e^{2\pi i (x \cdot \omega'-x' \cdot \omega)} \, dx d\omega.
	\end{equation}
	We observe that
	\begin{equation}
		\begin{aligned}
			\langle  H_{\eta_\gamma, \phi_{i(\gamma)}} f, f \rangle = \langle  \chi_{\eta_\gamma,} \V_{\phi_{i(\gamma)}}  f, \V_{\phi_{i(\gamma)}}  f \rangle
			= \int_{\eta_\gamma} |\V_{\phi_{i(\gamma)}} f (z)|^2 dz = \int_{\eta_\gamma,} |\V_{\gs}  f \natural\V_{\gs} \phi_{i(\gamma)} (z)|^2 dz . \label{eq4}
		\end{aligned}
	\end{equation}
	By the reverse triangle inequality we have
	\begin{equation}
		\begin{aligned}
			|\cV_{\gs}  f \natural\cV_{\gs} \phi_{i(\gamma)} (z)| \geq & |\cV_{\gs}f (z) |-|\cV_{\gs}f (z)-\cV_{\gs}  f \natural\cV_{\gs} \phi_{i(\gamma)} (z)|.
		\end{aligned}
	\end{equation}
	Since
	\begin{equation}
		\begin{aligned}
			|\cV_{\gs}f -\cV_{\gs}  f \natural\cV_{\gs} \phi_{i(\gamma)}| = |\cV_{\gs}f \natural (\cV_{\gs} \gs  -\cV_{\gs} \phi_{i(\gamma)})|
% 			= & |\cV_{\gs}f \natural (\cV_{\gs} \gs  -\cV_{\gs} \phi_i )(z))| 
			\leq \; |\cV_{\gs}f |\ast |\cV_{\gs} \gs -\cV_{\gs} \phi_{i(\gamma)} |
		\end{aligned}
	\end{equation}
	pointwise, we get
	\begin{equation}
		\begin{aligned}
		|\cV_{\gs}  f \natural\cV_{\gs} \phi_{i(\gamma)}(z)| \geq |\cV_{\gs}f (z) |- |\cV_{\gs}f |\ast |\cV_{\gs} \gs -\cV_{\gs} \phi_{i(\gamma)} |(z).
		\end{aligned}
	\end{equation}
	We proceed by inserting the previous expression and \eqref{eq4} into \eqref{Opinv}, hence
	\begin{align*}
		\sum_{\gamma \in \Gamma} \langle H_{\eta_\gamma, \phi_{i(\gamma)}} f, f \rangle \geq \|f\|_2^2
		- 2 \int_{\R^2} |\cV_{\gs} f (z)| \cdot ( |\cV_{\gs}f |\ast C_0 |\cV_{\gs} \gs|(z)) \, dz.
	\end{align*}
	Finally, since 
	\begin{align*}
		\int_{\R^2}  |\cV_{\gs}f (z)|\cdot \left( |\cV_{\gs}f |\ast |\cV_{\gs} \gs |(z) \right) \, dz
		%\leq \|\cV_{\gs}f\|_2 \cdot \||\cV_{\gs}f |\ast &|C_0\cV_{\gs} \gs |\|_2\\
		\leq \|\cV_{\gs}f\|_2 \cdot \|\cV_{\gs}f \|_2 \|\cV_{\gs} \gs\|_1
	\end{align*}
	we obtain
	\begin{equation*}
		\sum_{i \in \cI} \langle H_{\eta_\gamma, \phi_{i(\gamma)}}f, f \rangle \geq (1-2 \, C_0 \|  \V_{\gs} \gs \|_1) \cdot \|f\|_2^2 
	\end{equation*}
	which proves the invertibility of the operator sum as desired, whenever  $2 \, C_0 < \|  \V_{\gs} \gs \|_1$.
\end{proof}
\subsection{Multi-window Gabor frames}
In \cite{doro14} it was shown that the norm equivalence of the sum of localized functions as stated in Prop.~\ref{prop_norm_equivalence} and Prop.\ref{Prop_diffwin} is maintained if the full spectral representation of the operators is replaced by truncated versions. Here, we use a variant of this idea to construct families of local windows which eventually yield multi-window weaving frames. We need the following proposition.
\begin{proposition}
	Let  windows $\phi_i$ and symbols be chosen as in Prop.~\ref{prop_norm_equivalence} or  Prop.~\ref{Prop_diffwin}, respectively. Then we also have that  $ \sum_{\gamma \in \Gamma} \| H_{\eta_\gamma, \phi_{i(\gamma)}}^2 f\|_2^2\approx \| f\|_2^2$.  As a consequence, there exist constants $A$ and $B$ such that
	\begin{align}\label{normeq}
		A \| f\|_2^2 \leq \sum_{\gamma \in \Gamma} \| H_{\eta_\gamma, \phi_{i(\gamma)}}^2 f\|_2^2 \leq \sum_{\gamma \in \Gamma} \| H_{\eta_\gamma, \phi_{i(\gamma)}} f\|_2^2\leq  B \| f\|_2^2 \, .
	\end{align}
	If we choose $\varepsilon < \frac{A}{B}$, then the family of collections of all eigenfunctions $\phi_k^\gamma$ of $H_{\eta_\gamma, \phi_{i(\gamma)}}$, $\gamma \in \Gamma$, corresponding to eigenvalues bigger than $\varepsilon $ generates a frame, i.e.
	\begin{equation}
		\sum_{\gamma \in \Gamma} \| H^\varepsilon_{\eta_\gamma, \phi_{i(\gamma)}} f\|_2^2 \approx \sum_{\gamma \in \Gamma}\sum_{k: \, \l_k^\gamma > \varepsilon}  | \langle f, \phi_k^\gamma \rangle |^2\approx \norm{f}_2^2
	\end{equation}
%(\l_k^i)^2
\end{proposition}
\begin{proof}
	Since $\norm{H_{\eta_\gamma, \phi_{i(\gamma)}} f}_2 \leq \norm{H^\varepsilon_{\eta_\gamma, \phi_{i(\gamma)} }f}_2 + \varepsilon \norm{f}_2$ we obtain
	\begin{equation}
		\begin{aligned}
			\norm{H_{\eta_\gamma, \phi_{i(\gamma)}}^2 f}_2 \leq \norm{H^\varepsilon_{\eta_\gamma, \phi_{i(\gamma)}} H_{\eta_\gamma, \phi_{i(\gamma)}} f}_2 + \varepsilon \norm{H_{\eta_\gamma, \phi_{i(\gamma)}} f}_2
		\end{aligned}
	\end{equation}
	and, since $H_{\eta_\gamma, \phi_{i(\gamma)}}$ and $H^\varepsilon_{\eta_\gamma, \phi_{i(\gamma)}}$ commute, this yields
	\begin{align}
		\norm{H_{\eta_\gamma, \phi_{i(\gamma)}}^2 f}_2 \leq \norm{H^\varepsilon_{\eta_\gamma, \phi_{i(\gamma)}} f}_2 + \varepsilon \norm{H_{\eta_\gamma, \phi_{i(\gamma)}} f}_2
	\end{align}
	Taking sums and using the norm-equivalences in \eqref{normeq}, we obtain
	\begin{equation}
		A \| f\|_2^2\leq \sum_{\gamma \in \Gamma }\norm{H^\varepsilon_{\eta_\gamma, \phi_{i(\gamma)}} f}_2^2 + B \varepsilon	\| f\|_2^2,
	\end{equation}
which implies the claim. \end{proof}
The above proposition leads to the following result.
\begin{theorem}\label{thm_localization_frames}
	Let the windows $\phi_i$ be chosen as in Prop.~\ref{prop_norm_equivalence} or  Prop.~\ref{Prop_diffwin} and consider a well-spread family of symbols $\{\eta_\gamma\}_{\gamma \in \Gamma}$, such that $\sum_{\gamma \in \Gamma} \eta_\gamma (z) \approx 1$. Then, there exists an $\varepsilon>0$, such that  for each fixed $\phi_i$, $i \in \cI$, the family of collections of all eigenfunctions $\phi_k^{\gamma,i} $ of $H_{\eta_\gamma, \phi_i}$, $\gamma \in \Gamma$, corresponding to eigenvalues bigger than $\varepsilon $ generates a (multi-window Gabor) frame and all these frames are woven.
\end{theorem}
\section{Gaussian Windows and Elliptic Domains}\label{sec_Gaussian}
In this section we will have a look at localization operators on elliptic domains with an appropriate dilated Gaussian as carried out by Daubechies \cite{da88}. We denote the dilated standard Gaussian by
%\begin{equation}
$	\varphi_{0,L}(t) = 2^{1/4} \sqrt{L} \, e^{-\pi (Lt)^2}$.
%\end{equation}
The dilated standard Gaussian is essentially concentrated in an ellipse, which is best seen by computing
\begin{equation}
	\left| \V_{\varphi_{0,L}}\varphi_{0,L}(x,\omega) \right| = e^{-\frac{\pi}{2} \left(L^2 x^2 + \frac{\omega^2}{L^2}\right)}.
\end{equation}
Therefore, the ellipse
\begin{equation}
	\mathbb{E}_{L,R} = \left\lbrace (x,\omega) \in \R^2 \, | \, L^2 x^2 + \frac{\omega^2}{L^2} \leq R^2 \right\rbrace
\end{equation}
is the appropriate domain to be used for the localization operator. The eigenfunctions of the localization operator $H_{\mathbb{E}_{L,R},\varphi_{0,L}}$ are the dilated Hermite functions \cite{da88}. The eigenvalues are given by
\begin{equation}
	\l_{L,k}(R) = \l_k(R)= 1 - e^{-\pi R^2} \sum_{j=0}^k \frac{1}{j!} \, \left(\pi R^2\right)^j.
\end{equation}
We note that the eigenvalues depend on the size, but not the shape of the ellipse, whereas the eigenfunctions depend on the shape, but not the size of the ellipse. As a next step we compute how far two dilated Gaussians differ from each other in the $L^2$-norm. For $L_1,L_2 > 0$ we have
\begin{equation}
	\norm{g_{L_1}-g_{L_2}}_2^2 = 2 - 2\frac{\sqrt{2 L_1 L_2}}{\sqrt{L_1^2+L_2^2}}.
\end{equation}
If we want $\norm{g_{L_1} - g_{L_2}}_2 < \frac{1}{2}$, then
\begin{equation}
	\frac{64 - \sqrt{1695}}{49} < \frac{L_2}{L_1} <  \frac{64 + \sqrt{1695}}{49}.
\end{equation}
Numerically, this means that for $0.47 < \frac{L_2}{L_1} < 2.14$ we get woven multi-window Gabor frames consisting of sufficiently many Hermite functions. The results can be extended to \textit{chirped} or \textit{rotated} Gaussians. A chirped, dilated Gaussian is of the form
\begin{equation}\label{eq_chriped_Gauss}
	\varphi_{c,L}(t) = 2^{1/4} e^{\pi i c t^2} \sqrt{L} e^{-\pi (L t)^2}, \quad c,L > 0.
\end{equation}
We compute
\begin{equation}\label{eq_chirp_ambiguity}
	\left| \V_{\varphi_{c,L}}\varphi_{c,L}(x,\omega) \right| = e^{-\frac{\pi}{2} \left(\left(L^2+\frac{c^2}{L^2} \right)x^2 + 2 \frac{c}{L^2} x \omega + \frac{\omega^2}{L^2}\right)}.
\end{equation}
Therefore, a chirped Gaussian is essentially concentrated in a rotated ellipse described by the quadratic form in the exponent in \eqref{eq_chirp_ambiguity}. For more details on Gaussians and their concentration in phase space see \cite{fau16}.

\section{Conclusion}
We have established sufficient criteria for multi-window Gabor frames consisting of eigenfunctions of a localization operator to be woven. In particular we found out that two finite families of Hermite functions can constitute woven multi-window frames. However, there seems to be a gap between the necessary condition we know from the Balian-Low theorem, which is already sufficient for Gaussians, and the number of Hermite functions we need in our phase-space approach. Also, the problem posed in \cite{becagrlaly16} asks whether any families of rotated Gaussians yield woven Gabor frames. We have seen that if the difference of the Gaussians in the $\Lt[]$-norm is less than $1/2$, we get weaving frames by taking a finite number of generalized Hermite functions derived from the original Gaussians. As a next step, it would be interesting to show that the finite number is $1$ and we only need to take the Gaussians if the index set is sufficiently dense  in phase-space, in particular in the case of a lattice with density greater than $1$.


% \section*{Acknowledgment}
% M.D.\ was supported by the Vienna Science and Technology Fund (WWTF): project SALSA (MA14-018); M.F.\ was supported by the Austrian Science Fund (FWF): [P26273-N25] and [P27773-N23].



%\bibliographystyle{IEEEtran}

\begin{thebibliography}{10}

\bibitem{becagrlaly16}
T.~Bemrose, P.G.~Casazza, K.~Gr\"ochenig, M.C.~Lammers, R.G.~Lynch.
\newblock{\em Weaving Frames}.
\newblock{Operators and Matrices},
10(4):1093-1116, (2016)

\bibitem{caly15}
P.G.~Casazza, R.G.~Lynch. 
\newblock{\em Weaving properties of Hilbert space frames}. 
\newblock{2015 International Conference on Sampling Theory and Applications (SampTA), Washington, DC},(2015)

\bibitem{da88}
I.~Daubechies.
\newblock{\em Time-frequency localization operators: a geometric phase space approach}.
\newblock{Information Theory, IEEE Transactions on},
34(4):605-612, (1988)

\bibitem{doro14}
M.~D\"orfler, J.L.~Romero.
\newblock{\em Frames adapted to a phase-space cover}.
\newblock{Constructive Approximation},
39(3):445-484, (2014)

\bibitem{fau16}
M.~Faulhuber.
\newblock{\em Gabor frame sets of invariance: a Hamiltonian approach to Gabor frame deformations}.
\newblock{Journal of Pseudo-Differential-Operators and Applications},
7(2):213-235, (2016)

\bibitem{zezi97}
M.~Zibulski, Y.~Y. Zeevi. 
\newblock{\em Analysis of Multiwindow Gabor-Type Schemes by Frame Methods}. 
\newblock{Applied and Computational Harmonic Analysis}, 4(2):188-221, (1997)
% \bibitem{feno03}
% H.G.~Feichtinger, K.~Nowak
% \newblock{\em A First Survey of Gabor Multipliers}.
% \newblock{in Feichtinger, Strohmer (eds.), Advances in Gabor Analysis},
% \newblock{Applied and Numerical Harmonic Analysis},
% pp.~99-128, Birkh\"auser Boston, (2003)

\end{thebibliography}

\end{document}